\documentclass[12pt,a4paper,reqno]{amsart}
\usepackage{amsmath}
\usepackage{amssymb}
\usepackage{amsfonts}
\usepackage[matrix,arrow]{xy}
\usepackage{pb-diagram,pb-xy}
\usepackage{graphicx}
\newtheorem{thm}{Theorem}[section]
\newtheorem{lemma}[thm]{Lemma}

\newtheorem{prop}[thm]{Proposition}

\newtheorem{remark}[thm]{Remark}

\newcommand{\LoopB}[1]{{\Omega}_{B}(#1)}
\newcommand{\cover}[1]{\widetilde{#1}}

\newcommand{\plocal}[1]{{#1}_{(p)}}
\newcommand{\aplocal}[1]{{#1}^B_{(p)}}

\newcommand{\aqlocal}[1]{{#1}^B_{(q)}}
\newcommand{\Plocal}[1]{{#1}_{(P)}}
\newcommand{\Pnlocal}[1]{{#1}_{P}}
\newcommand{\aPlocal}[1]{{#1}^B_{(P)}}
\newcommand{\aPnlocal}[1]{{#1}^B_{P}}
\newcommand{\Pblocal}[1]{{#1}_{(Q)}}
\newcommand{\Pbnlocal}[1]{{#1}_{Q}}
\newcommand{\aPblocal}[1]{{#1}^B_{(Q)}}
\newcommand{\aPbnlocal}[1]{{#1}^B_{Q}}

\newcommand{\zlocal}[1]{{#1}_{(0)}}
\newcommand{\azlocal}[1]{{#1}^B_{(0)}}

\newcommand{\integer}{\mathbb{Z}}

\newcommand{\sq}{{\mathcal S\kern-0.2em q}}

\newcommand{\comp}{\smash{\lower-.1ex\hbox{\scriptsize$\circ$}}}
\title{\sc{Co-H-spaces and Almost Localization}}

\author{Cristina ~Costoya}
\address[C.~Costoya]{Departamento de Computaci\'on, \'Alxebra,
Universidade da Coru{\~n}a, Campus de Elvi{\~n}a, s/n, 15071 - A Coru{\~n}a, Spain.}
\email{cristina.costoya@udc.es}

\author{Norio Iwase}
\address[N.~Iwase]{Faculty of Mathematics, Kyushu University,  Fukuoka 810-8560, Japan}
\email{iwase@math.kyushu-u.ac.jp}
\begin{document}
\maketitle
%
%
\begin{abstract}
Apart from simply-connected spaces, a non simply-connected co-H-space is a typical example of a space $X$ with a co-action of $B\pi_1(X)$ along $r^X : X \rightarrow B\pi_{1}(X)$  the classifying map of the universal covering. If such a space $X$ is actually a co-H-space, then the fibrewise $p$-localization of $r^X$ (or the `almost' $p$-localization of $X$) is a fibrewise co-H-space (or an `almost' co-H-space, resp.) for every prime $p$.  In this paper, we show that the converse statement is true, i.e., for a non simply-connected space $X$ with
a co-action of $B\pi_1(X)$ along $r^X$, $X$ is a co-H-space if, for every prime $p$, the almost $p$-localization of $X$ is an almost co-H-space.
\end{abstract}
%
%

\section{Fundamentals and Results}\label{fundamentals}

We assume that spaces have the homotopy type of CW complexes and are based, and maps and homotopies preserve base points. A space $X$ is a co-H-space if there exists a comultiplication, say $\mu^X: X \rightarrow X \vee X$ satisfying $j^X \circ \mu^X\simeq \Delta^X$, where $j^X: X\vee X \hookrightarrow X \times X$ is the inclusion and $\Delta^X : X \rightarrow X \times X$ is the diagonal.
Similar to spaces, we say that a group $G$ is a co-$H$-group if there exists a homomorphism $G \rightarrow G \ast G$ so that the composition with the first and second projections are the identity of $G$.
Thus the fundamental group of a co-H-space is a co-$H$-group and the classifying space of a co-H-group is a co-H-space.

When $X$ is a simply-connected co-H-space, the $p$-localization $\plocal{X}$ is also a co-H-space for any prime $p$.
The immediate problem is whether the converse statement holds.
In \cite{Costoya:loc_co-hopf}, the first author settles the answer in positive when $X$ is a finite simply-connected complex.
In this paper, we extend the above result to non simply-connected spaces in terms of  fibrewise $p$-localization (see \cite{Bendersky:partial-localization}, \cite{May:fibrewise-local}) or paraphrasing the second author, in terms of  almost $p$-localization (see \cite{Iwase:counter-ganea}), rather than a usual $p$-localization since the only nilpotent non simply-connected co-H-space is the circle \cite{HMR:co-loop}.

From now on, we assume that a space $X$ is the total space of a fibrewise space $r^X : X \rightarrow B\pi_1(X)$, where $r^{X}$ is the classifying map of $c^X : \widetilde{X} \to X$, the universal covering of $X$.

A co-action of $B\pi_1(X)$ along $r^X$ is a map $\nu: X \rightarrow B\pi_1(X) \vee X$ such that when composed with the first projection $B\pi_1(X) \vee X \rightarrow B\pi_{1}(X)$ we obtain $r^X$ and composed with the second projection $B\pi_1(X) \vee X \rightarrow X$ we obtain the identity $1_X$.
In other words, it consists of a copairing $\nu: X \rightarrow B\pi_{1}(X) \vee X$ with coaxes $r^X$ and $1_X$ in the sense of Oda \cite{Oda:co-pairing}.
Since the fundamental group of a space $X$ with a co-action of $B\pi_1(X)$ along $r^X$ is clearly a co-$H$-group, $\pi_1(X)$ is free by \cite{EG:cat1} or \cite{ISS:hom_co-h}, and hence $B\pi_1(X)$ has the homotopy type of a bunch of circles, say $B$. 
Thus, $X$ is a fibrewise space over a bunch of circles $B = B\pi_1(X)$.
Let $s^X : B \rightarrow X$ represent the generators of $\pi_1(X)$ associated with circles.
We may assume that $r^X\circ s^X \simeq 1_B$ and so $X$ is a fibrewise-pointed space over $B$.

For such a space admitting a co-action of $B$  along $r^X : X \rightarrow B$, a retraction map $\rho: B \vee D(X) \rightarrow X$ is constructed in \cite[Theorem 3.3.]{ISS:hom_co-h}, where $D(X)$ a simply-connected finite complex. If we factorize $\rho_{\mid_ D(X)} : D(X) \rightarrow X$ through $  \widetilde X $, we deduce that if there exists a co-action of $B$ along $r^X$, then $X$ is dominated by $B \vee \widetilde X$ .
The space $B$ is a co-H-space and, if $\widetilde X$ is  a co-H-space, then $X$ is dominated by a co-H-space hence a co-H-space itself.  On the other hand, recall  that co-H-spaces are spaces with Lusternik Schnirelmann category, L-S cat, less than or equal to one. Since the L-S cat of $\widetilde X$ can not exceed the L-S cat of $X$ \cite{Fox:ls-cat}
if $X$ is a co-H-space then $\widetilde X$ is also a co-H-space. Gathering these results altogether and, considering that the almost localization is natural, we obtain the following result.

\begin{prop}\label{prop:domination}
Let $r^{X} : X \to B\pi_{1}(X)$ be a fibrewise space over $B = B\pi_{1}(X)$ with a co-action of $B$ along $r^{X}$.
Then, the following two statements hold.
\begin{enumerate}
\item The space $X$ is a co-H-space if and only if $\widetilde X$ is a co-H-space.
\item The almost $p$-localization of $X$ for a prime $p$, $X^B_{(p)}$,  is also a fibrewise space over $B$ with a co-action of $B$ along
the classifying map of the universal covering.
\end{enumerate}
\end{prop}

Following earlier authors, e.g., \cite{Henn:almost-rat,HI:complete-ganea,Iwase:counter-ganea,ISS:hom_co-h}, we paraphrase the fibrewise property (see \cite{James:over-B,James:intro-over-B}) for a fibrewise based space $r^X : X \rightarrow  B\pi_1(X)$ as `almost' property for a space $X$. Recall that $X$ is an `almost' co-$H$-space if there exists 
a map $\mu: X \rightarrow X {\vee}_B X$,  such that when composed with each projection, the identity of $X$ is obtained, where $ X {\vee}_B  X$ is the pushout of the folding map $\nabla_B:  B \vee B \rightarrow B$ and $s^X \vee s^X:  B \vee B \rightarrow X \vee X$,  $s^X$ section of the classifying map.

Our main result is the following theorem. The rest of this paper will be devoted  to prove it.  
\begin{thm}\label{thm:main}
Let $r^{X} : X \to B\pi_{1}(X)$ be a fibrewise space over $B= B\pi_{1}(X)$ with a co-action of $B$ along $r^{X}$. 
If $X$ is a connected finite complex whose almost $p$-localization is an almost co-H-space for every prime $p$, then $X$ is a co-$H$-space.
\end{thm}

Some notation is required. For a set of primes $P$, the almost $P$-localization of $r^X : X \rightarrow B$, in other words a fibrewise $P$-localization, is a map $l^B_{(P)}: X \rightarrow X_{(P)}^B$ which commutes with projections to $B$. The map
$l^B_{(P)}$ induces an isomorphism of fundamental groups and acts as a standard $P$-localization on the fibre $\widetilde X$, so $\widetilde {X_{(P)}^B} = \widetilde X_{(P)}$.

Most of the proofs in this paper follow by induction on the dimension of the space. Henceforth, it will be useful to work with the almost localization introduced by the second author in \cite{Iwase:counter-ganea}.

\section{`Almost' version of Zabrodsy Mixing}\label{sect:preliminaries}

To show the main result, we need a fibrewise version of a result of Zabrodsky \cite[Proposition 4.3.1]{Zabrodsky:Hopf-space}.
Let $\Pi$ denote the set of all primes, $P$ and $Q$ disjoint sets with $P \sqcup Q = \Pi$. We first give a lemma.

\begin{lemma}\label{fibrewisepq} Let $M = B \vee M_0$ be the wedge-sum of $B $, a bouquet of $1$-dimensional spheres, and $M_0$ a simply-connected finite CW-complex. Let $X$ be fibrewise pointed space over $B$. 
Then, $$[M, X]_B = \rm{pullback}  \big \{{[M, X^B_{(P)}]}_B \rightarrow  {[ M, X^B_{(0)}]}_B   \leftarrow {[ M, X^B_{(Q)}]}_B   \big \}$$
\end{lemma}

\begin{proof} Since we have the following equivalence between homotopy sets,
$${[M, X^B_{(P)}]}_B ={[B \vee M_0 , X^B_{(P)}]}_B  ={[ M_0 , X^B_{(P)}]} = [ M_0 , \widetilde {X^B_{(P)}}] = [ M_0 , \widetilde {X}_{(P)}], $$
it is equivalent to proof that 
$$[M_0, \widetilde X] = \rm{pullback}  \{{[M_0, \widetilde X_{(P)}]} \rightarrow  {[ M_0, \widetilde X_{(0)}]}   \leftarrow {[ M_0, \widetilde X_{(Q)}]}   \}.$$
This is clear by the fracture square lemma of Bousfield-Kan \cite[6.3.(ii)]{BK:localization}, since $M_0$ is finite and $\widetilde X$ is simply-connected.
\end{proof}

\begin{prop}\label{prop:decomp}
Let  $M$ satisfy the conditions of the previous lemma and $f : M \to X$ be a fibrewise pointed rational equivalence of fibrewise pointed spaces over $B=B\pi_{1}(M)=B\pi_{1}(X)$ commuting with co-actions of $B$ along the classifying maps of the universal coverings $r^M$ and $ r^X$ respectively. We then have:
\begin{enumerate}
\item\label{prop:decomp;existence}
There exists a unique fibrewise pointed space $M({P},f)$ over $B$ and fibrewise maps $\psi({P},f) : M \to M({P},f)$ and $\phi({Q},f) : M({P},f) \to X$, such that $\psi({P},f)$ is a fibrewise $P$-equivalence, $\phi({Q},f)$ is a fibrewise ${Q}$-equivalence and the following diagram is  homotopy commutative:
\begin{equation*}
\begin{diagram}
\node{M}
	\arrow[2]{e,t}{f}
	\arrow{se,b}{\psi({P},f)}
\node{}
\node{X}
\\
\node{}
\node{M({P},f)}
	\arrow{ne,b}{\phi({Q},f)}
\node{}
\end{diagram}
\end{equation*}
\item \label{prop:decomp;naturality}
The fibrewise pointed space $M({P},f)$ over $B$ and fibrewise maps $\psi({P},f)$ and $\phi({Q},f)$ are natural with respect to $f$, i.e., for a homotopy commutative square:
\begin{equation*}
\begin{diagram}
\node{M_{1}}
	\arrow[2]{e,t}{f_{1}}
	\arrow{s,t}{s}
\node{}
\node{X_{1}}
	\arrow{s,b}{t}
\\
\node{M_{2}}
	\arrow[2]{e,b}{f_{2}}
\node{}
\node{X_{2}}
\end{diagram}
\end{equation*}
with $f_{1}$ and $f_{2}$ fibrewise rational equivalences, there exists $k({P};f_{1},f_{2}) : M({P},f_{1}) \to M({P},f_{2})$ such that the following diagram is homotopy commutative:
\begin{equation*}
\begin{diagram}
\node{M_{1}}
	\arrow[2]{e,t}{\psi({P},f_{1})}
	\arrow{s,t}{s}
\node{}
\node{M({P},f_{1})}
	\arrow[2]{e,t}{\phi({Q},f_{1})}
	\arrow{s,b}{k({P};f_{1},f_{2}) }
\node{}
\node{X_{1}}
	\arrow{s,b}{t}
\\
\node{M_{2}}
	\arrow[2]{e,b}{\psi({P},f_{2})}
\node{}
\node{M({P},f_{2})}
	\arrow[2]{e,b}{\phi({Q},f_{2})}
\node{}
\node{X_{2}}
\end{diagram}
\end{equation*}
\item
If further $M$ and $X$ have the integral homology of a finite space, then so has $M(P,f)$.
\end{enumerate}
\end{prop}
\begin{proof}

We define $\alpha({P},f) = \aPnlocal{\ell}\circ\aPlocal{f} = \azlocal{f}\circ\aPnlocal{\ell} : \aPlocal{M} \to \azlocal{X}$ and $\beta({Q},f) = \aPbnlocal{\ell} : \aPblocal{X} \to \azlocal{X}$, where by $\aPnlocal{\ell} : Y \to \azlocal{Y}$ we denote the almost rationalization of a fibrewise based ${P}$-local space  $r^Y : Y \to B$. With the above maps, we obtain the following commutative diagram:
\begin{equation}\label{MZabrodsky}
\begin{diagram}
\node{\aPblocal{M}} 
 \arrow{e,t}{\aPblocal{f}}
 \arrow{s,l}{\aPbnlocal{\ell}}
\node{\aPblocal{X}}
 \arrow{e,=}
 \arrow{s,l}{\beta({Q},f)}
\node{\aPblocal{X}}
 \arrow{s,l}{\aPbnlocal{\ell}}
\\
\node{\azlocal{M}}
 \arrow{e,t}{\azlocal{f}}
\node{\azlocal{X}}
 \arrow{e,=}
\node{\azlocal{X}}
\\
\node{\aPlocal{M}}
 \arrow{e,=}
 \arrow{n,l}{\aPnlocal{\ell}}
\node{\aPlocal{M}}
 \arrow{e,b}{\aPlocal{f}}
 \arrow{n,l}{\alpha({P},f)}
\node{\aPlocal{X}}
 \arrow{n,l}{\aPnlocal{\ell}}
\end{diagram}
\end{equation}
Now, taking the universal covering spaces, we get pullback diagrams by using Bousfield-Kan's fibre square of localizations \cite[p. 127]{BK:localization}:
$$
\begin{diagram}
\node{\cover{M}}
 \arrow{e,t}{\Pblocal{\cover{\ell}}}
 \arrow{s,l}{\Plocal{\cover{\ell}}}
\node{\Pblocal{\cover{M}}}
 \arrow{s,r}{\Pbnlocal{\cover{\ell}}}
\\
\node{\Plocal{\cover{M}}}
 \arrow{e,b}{\Pnlocal{\cover{\ell}}}
\node{\zlocal{\cover{M}}}
\end{diagram}
\qquad
\begin{diagram}
\node{\cover{X}}
 \arrow{e,t}{\Pblocal{\cover{\ell}}}
 \arrow{s,l}{\Plocal{\cover{\ell}}}
\node{\Pblocal{\cover{X}}}
 \arrow{s,r}{\Pbnlocal{\cover{\ell}}}
\\
\node{\Plocal{\cover{X}}}
 \arrow{e,b}{\Pnlocal{\cover{\ell}}}
\node{\zlocal{\cover{X}}}
\end{diagram} 
$$
Then, by Theorem 6.3 in \cite{Dold:fibration}, we obtain that the following square diagrams are fibrewise homotopy equivalent to fibrewise pullback diagrams:
$$
\begin{diagram}
\node{M}
 \arrow{e,t}{\aPblocal{\ell}}
 \arrow{s,l}{\aPlocal{\ell}}
\node{\aPblocal{M}}
 \arrow{s,r}{\aPbnlocal{\ell}}
\\
\node{\aPlocal{M}}
 \arrow{e,b}{\aPnlocal{\ell}}
\node{\azlocal{M}}
\end{diagram}
\qquad
\begin{diagram}
\node{X}
 \arrow{e,t}{\aPblocal{\ell}}
 \arrow{s,l}{\aPlocal{\ell}}
\node{\aPblocal{X}}
 \arrow{s,r}{\aPbnlocal{\ell}}
\\
\node{\aPlocal{X}}
 \arrow{e,b}{\aPnlocal{\ell}}
\node{\azlocal{X}}
\end{diagram}
$$
Let the pair of maps $\hat\alpha({P},f) : M({P},f) \to \aPblocal{X}$ and $\hat\beta({Q},f) : M({P},f) \to \aPlocal{M}$ be the fibrewise pullback of $\alpha({P},f)$ and $\beta({Q},f)$ so  that $\aPlocal{M({P},f)} =\aPlocal{M}$ and $\aPblocal{M({P},f)} =\aPblocal{X}$:
$$
\begin{diagram}
\node{M({P},f)}
 \arrow{e,t}{\hat\alpha({P},f)}
 \arrow{s,l}{\hat\beta({Q},f)}
\node{\aPblocal{X}}
 \arrow{s,r}{\beta({Q},f)}
\\
\node{\aPlocal{M}}
 \arrow{e,b}{\alpha({P},f)}
\node{\azlocal{X}}
\end{diagram}
$$
Taking the fibrewise pullback of the vertical arrows in diagram (\ref{MZabrodsky}), we immediately obtain the existence of maps
 $\psi({P},f) : M \to M({P},f)$ and $\phi({Q},f) : M({P},f) \to X$ which fit within the following commutative diagram:
$$
\begin{diagram}
\node{\aPblocal{M}}
 \arrow{e,t}{\aPblocal{f}}
\node{\aPblocal{X}}
 \arrow{e,=}
\node{\aPblocal{X}}
\\
\node{M}
 \arrow{e,t}{\psi({P},f)}
 \arrow{n,l}{\aPblocal{\ell}}
 \arrow{s,l}{\aPlocal{\ell}}
\node{M({P},f)}
 \arrow{e,t}{\phi({Q},f)}
 \arrow{n,l}{\hat\alpha({P},f)}
 \arrow{s,l}{\hat\beta({Q},f)}
\node{X}
 \arrow{n,l}{\aPblocal{\ell}}
 \arrow{s,l}{\aPlocal{\ell}}
\\
\node{\aPlocal{M}}
 \arrow{e,=}
\node{\aPlocal{M}}
 \arrow{e,b}{\aPlocal{f}}
\node{\aPlocal{X}}
\end{diagram}
$$
We then have that $f$ and $\phi({Q},f)\comp\psi({P},f)$ have the same almost $P$-localization and $Q$-localization, hence they are equal by Lemma \ref{fibrewisepq}. Since $f$ is a rational equivalence, $\aPlocal{f}$ is a $Q$-equivalence and hence so is $\phi({Q},f)$.
Similarly we obtain that $\psi({P},f)$ is a ${P}$-equivalence.
Thus 1) in Proposition \ref{prop:decomp;naturality} is proved.
The verification of 2) and 3) is straightforward and we leave it to the readers.
\end{proof}

\begin{prop}\label{prop:p-local}
For any prime $p$, we have $M(\bar p,f) = M( \bar p,\aplocal{\ell}{\comp}f)$.
\end{prop}
\begin{proof}
By the unique existence in \ref{prop:decomp;existence}) of Proposition \ref{prop:decomp}, the proposition follows  from the diagram
$\begin{diagram}
\node{M}
	\arrow{e,t}{\psi(\bar p,f)}
\node{M(\bar p,f)}
	\arrow{e,t}{\aplocal{\ell}{\comp}\phi(p,f)}
\node{\aplocal{X},}
\end{diagram}$
since $\aplocal{\ell}$ is a $ p$-equivalence.
\end{proof}

We now give a result that will be used in the proof of the main theorem.
\begin{prop}\label{prop:intermediatecoH} Let $f : M \to X$ be a fibrewise pointed rational equivalence satisfying the conditions of the previous proposition. If  $f$ is a fibrewise co-$H$-map, then  $\psi ( P, f): M \rightarrow M (P, f)$ is a fibrewise co-$H$-map.
\end{prop}

\begin{proof} We consider the following diagram, where $\rho_{k}^{Y}$ is the projection onto $k$-th factor.
Then, the composition of the vertical arrows on the left, and the composition of the vertical arrows on the right, are the identity of $M$ and $X$, respectively.
The existence of the dotted arrow is given by \ref{prop:decomp;existence}) in Proposition \ref{prop:decomp}:
$$
\begin{diagram}
\node{M}
	\arrow[2]{e,t}{\psi({P},f)}
	\arrow{s,t}{\nu_M}
\node{}
\node{M({P},f)}
	\arrow[2]{e,t}{\phi({Q},f )}
	\arrow{s,b,..}{\nu_{M(P,f)}}
\node{}
\node{X}
	\arrow{s,b}{\nu_X}
\\
\node{M\vee_B M}
	\arrow[2]{e,b}{\psi({P},f) \vee \psi({P},f)  }
	\arrow{s,l}{\rho_{k}^{M}}
\node{}
\node{M({P},f) \vee_B M({P},f)}
	\arrow[2]{e,b}{\phi({Q},f) \vee \phi({Q},f) }
	\arrow{s,r}{\rho_{k}^{M({P},f)}}
\node{}
\node{X\vee_B X}
\arrow{s,r}{\rho_{k}^{X}}
\\
\node{M}
	\arrow[2]{e,t}{\psi({P},f)}
	\node{}
\node{M({P},f)}
	\arrow[2]{e,t}{\phi({Q},f )}
	\node{}
\node{X}
	\end{diagram}
$$
In order to see that ${M(P,f)}$ is a fibrewise co-$H$-space, it is enough to show that $\rho_{k}^{M({P},f)} \circ\nu_{M({P},f)}$ is a fibrewise homotopy equivalence for $k=1, 2$.
By the commutativity of the diagram, we can easily see that $\rho^{M({P},f)}_k\circ \nu_{M({P},f)}$ induces both a homology $\mathbb{Z}_{(P)}$-equivalence and a homology $\mathbb{Z}_{(Q)}$-equivalence on the fibre.
Since the fibre of $M(P,f) \to B$ is simply-connected, $\rho_{k}^{M({P},f)}\circ \nu_{M({P},f)}$ induces a homotopy equivalence on the fibre.
Then by \cite[Theorem 6.3]{Dold:fibration}, we can conclude that $\rho_{k}^{M({P},f)}\circ \nu_{M({P},f)}$ is a fibrewise homotopy equivalence.
Thus $M(P,f)$ is an almost co-H-space and the difference between $\rho_{k}^{M({P},f)} \circ \nu_{M({P},f)}$ and the identity of $M(P,f)$ is defined as a fibrewise map $d : M(P,f) \to M(P,f)$.
Again using $P$ and $Q$-localization, we see that $d$ is trivial and hence $\rho_{k}^{M({P},f)}\circ \nu_{M({P},f)}$ is homotopic to the identity.
Thus $\psi(P,f) : M \to M(P,f)$ is a fibrewise co-H-map.
\end{proof}

From now on, we fix a space $X$ which is a finite complex with  a co-action of $B=B\pi_1(X)$ along $r^{X} : X \to B$ the classifying map of the universal covering of $X$.
 Hence $X$ is a fibrewise pointed space as we mentionned in Section \ref{fundamentals}.
By \cite{Iwase:counter-ganea}, a co-H-space $X$ has a homology decomposition $\{X_{i}\,;\,i\geq1\}$ such that
\begin{equation*}
B = X_{1} \subseteq X_{2} \subseteq \cdots \subseteq X_{n} = X,
\end{equation*}
together with a cofibration sequences
$
S_{i} \xrightarrow{h_{i}} X_{i} \hookrightarrow X_{i+1}, \ i \geq 1
$, 
where $S_{i}$ stands for a Moore space of type $(H_{i+1}(X),i)$.
\section{Fibrewise co-H-structures on almost localizations}
By the assumption on a finite complex $X$, $\aplocal{X}$ is an almost $p$-local co-H-space, which also implies that $\azlocal{X}$ is an almost rational co-H-space.
An almost rational co-H-space $\azlocal{X}$ is a wedge-sum of $B$ and  finitely many rational spheres of dimension $\geq 2$ \cite{Henn:almost-rat}. Hence the $k$'-invariants are all of finite order
and we can prove the following lemma.
\begin{lemma}\label{lem:rational}
There is a fibrewise pointed space $M=M(X)$ which is a wedge-sum of a finite number of spheres and, an almost rational equivalence $f = f(X) : M \to X$ which satisfies, for any prime $p$, that there are fibrewise co-$H$-structures on $M$ such that $\aplocal{\ell}{\comp}f : M \to \aplocal{X}$ is a fibrewise co-$H$-map.
\end{lemma}
\begin{proof}
We construct $M$ and $f$ by induction on the homological dimension of $X$.
When $X=X_{1}$, we have nothing to do.
So we may assume that $X=X_{i+1}$ and that we have constructed $M_{i}=M(X_{i})$ and a fibrewise rational equivalence $f_{i}=f(X_{i}) : M_{i} \to X_{i}$ satisfying the lemma.
Let $d_{i}$ be the order of the $k$'-invariant $h_{i} : S_{i} \to X_{i}$ and $M_{i+1} = M_{i} \vee {\Sigma}S^{0}_{i}$ where $S^{0}_{i} \subseteq S_{i}$ is the Moore space of type $(H_{i+1}(X;\integer)/\textrm{torsion},i)$.
We denote the multiplication by $d_{i}$ by $d_{i} : S_{i} \to S_{i}$.
Let $\nu_{p} : \aplocal{X} \to \aplocal{X} \vee_{B} \aplocal{X}$ be the given co-$H$-structure on $\aplocal{X}$ and $\nu^{M}_{i} : M_{i} \to M_{i} \vee_{B} M_{i}$ the co-$H$-structure such that $\aplocal{\ell}{\comp}f_{i}$ is a fibrewise co-$H$-map.
Since $\aplocal{(X_{i})}$ gives the homology decomposition of $\aplocal{X}$, $\nu_{p}$ induces a fibrewise co-$H$-structure $\nu^{i}_{p}$ on $\aplocal{(X_{i})}$ by the arguments of \cite{Iwase:counter-ganea}.
Recall that the $k$'-invariants are all of finite order, hence $h_{i}{\comp}d_{i}\vert_{S^{0}_{i}} = \ast$ obtaining the following commutative diagrams 
$$
\begin{diagram}
\node{S^{0}_{i}}
	\arrow{e,t}{\ast}
	\arrow{s,t}{d_{i}\vert_{S^{0}_{i}}}
\node{M_{i}}
	\arrow{s,b}{f_{i}}
\node{M_{i}}
	\arrow{s,t}{\aplocal{\ell}{\comp}f_{i}}
	\arrow{e,t}{\nu^{M}_{i}}
\node{M_{i} \vee_{B} M_{i}}
	\arrow{s,b}{\aplocal{\ell}{\comp}f_{i} \vee \aplocal{\ell}{\comp}f_{i}}
\\
\node{S_{i}}
	\arrow{e,b}{h_{i}}
\node{X_{i}}
\node{\aplocal{(X_{i})}}
	\arrow{e,b}{\nu^{i}_{p}}
\node{\aplocal{(X_{i})} \vee_{B} \aplocal{(X_{i})}}
\end{diagram}
$$
By taking cofibres of the horizontal arrows of the left-hand-side square, we get the following commutative diagram:
\begin{equation*}
\begin{diagram}
\node{M_{i}}
	\arrow{s,t}{f_{i}}
	\arrow{e,J}
\node{M_{i+1}}
	\arrow{s,b}{f'_{i+1}}
	\arrow{e}
\node{{\Sigma}S^{0}_{i}}
	\arrow{s,b}{{\Sigma}d^{0}_{i}\vert_{{\Sigma}S^{0}_{i}}}
\\
\node{X_{i}}
	\arrow{e,J}
\node{X_{i+1}}
	\arrow{e}
\node{{\Sigma}S_{i},}
\end{diagram}
\end{equation*}
where $M_{i+1} = M_{i} \vee {\Sigma}S^{0}_{i}$ and $f'_{i+1}$ is a rational equivalence, since  $f_{i}$ and ${\Sigma}d^{0}_{i}\vert_{{\Sigma}S^{0}_{i}}$ are so.

Let $\nu'^{M}_{i+1}$ be a fibrewise co-$H$-structure defined by $\nu'^{M}_{i+1}\vert_{M_{i}}=\nu^{M}_{i}$ and $\nu'^{M}_{i+1}\vert_{{\Sigma}S^{0}_{i}} : {\Sigma}S^{0}_{i} \overset{\nu^{S}_{i}}\to {\Sigma}S^{0}_{i}\vee{\Sigma}S^{0}_{i} \subseteq M_{i+1} \vee_{B} M_{i+1}$, where $\nu^{S}_{i}$ denotes the standard co-$H$-structure of ${\Sigma}S^{0}_{i}$.
$$
\begin{diagram}
\node{M_{i+1}}
	\arrow{s,t}{\aplocal{\ell}{\comp}f'_{i+1}}
	\arrow{e,t}{\nu'^{M}_{i+1}}
\node{M_{i+1} \vee_{B} M_{i+1}}
	\arrow{s,b}{\aplocal{\ell}{\comp}f'_{i+1} \vee_{B} \aplocal{\ell}{\comp}f'_{i+1}}
\\
\node{\aplocal{X_{i+1}}}
	\arrow{e,b}{\nu^{i+1}_{p}}
\node{\aplocal{X_{i+1}} \vee_{B} \aplocal{X_{i+1}}}
\end{diagram}
$$
The difference between $(\aplocal{\ell}{\comp}f'_{i+1} \vee_{B} \aplocal{\ell}{\comp}f'_{i+1}){\comp}\nu'^{M}_{i+1}$ and $\nu^{i+1}_{p}{\comp}\aplocal{\ell}{\comp}f'_{i+1}$ is given by a map $ D_{i+1}(p) : {\Sigma}S^{0}_{i} \to \LoopB{\aplocal{X_{i}}}{\ast}_{B}\LoopB{\aplocal{X_{i}}} \subseteq \LoopB{\aplocal{X_{i+1}}}{\ast}_{B}\LoopB{\aplocal{X_{i+1}}}$.
Since $f'_{i+1}$ is a rational equivalence, a multiple $a_{i+1}(p)D_{i+1}(p)$ can be pulled back to $\LoopB{M_{i}}{\ast}_{B}\LoopB{M_{i}}$.
Now, the $k$'-invariants are  of finite order so $X_{i+1}$ and $M_{i+1}$ have the same almost $p$-localization except for primes in $P$, a finite set of primes.
Let $a_{i+1}$ be the multiple of $a_{i+1}(p)$'s for all $p \in {P}$.
Let $f_{i+1} : M_{i+1} \to X_{i+1}$ be a map  defined by $f_{i+1}\vert_{M_{i}}=f_{i}$ and by $f_{i+1}\vert_{{\Sigma}S^{0}_{i}}=a_{i+1}f'_{i+1}\vert_{{\Sigma}S^{0}_{i}} : {\Sigma}S^{0}_{i} \to X_{i+1}$.
Then we choose a map $D^{M}_{i+1}(p) : {\Sigma}S^{0}_{i} \to \LoopB{M_{i}}{\ast}_{B}\LoopB{M_{i}}$ as a pull-back of $a_{i+1}D_{i+1}(p)$ onto $\LoopB{M_{i}}{\ast}_{B}\LoopB{M_{i}}$.
The co-$H$-structure defined by $\nu^{M}_{i+1}\vert_{M_{i}}=\nu^{M}_{i}$ and by $\nu^{M}_{i+1}\vert_{{\Sigma}S^{0}_{i}}=\nu'^{M}_{i+1}\vert_{{\Sigma}S^{0}_{i}}+D^{M}_{i+1} : {\Sigma}S^{0}_{i} \to M_{i+1} \vee_{B} M_{i+1}$,  together with  $\nu^{p}_{i+1}$ verify that  $\aplocal{\ell}{\comp}f_{i+1}$ is a fibrewise co-$H$-map at any prime $p \in {P}$.
For a prime $q \not\in {P}$, we may define comultiplication of $\aqlocal{(X_{i+1})}$ by that of $\aqlocal{(M_{i+1})}$, since $\aqlocal{(X_{i+1})}=\aqlocal{(M_{i+1})}$.
Hence the induction holds and we obtain the lemma.
\end{proof}
\section{Proof of Theorem \ref{thm:main}}\label{sect:proof-main}

Let ${P}_{1}$ and ${P}_{2}$ be disjoint sets of primes.
Then we have a commutative diagram 
\begin{equation}\label{diag:pushout}
\begin{diagram}
\node{M}
	\arrow[2]{e,t}{\psi({\overline P}_{1},f)}
	\arrow{s,t}{\psi({\overline P}_{2},f)}
\node{}
	\arrow{e}
\node{M({\overline P}_{1},f)}
	\arrow{s,b}{\psi({\overline P}_{1}{\cap}{\overline P}_{2},\phi({ P}_{1},f))}
\\
\node{M({\overline P}_{2},f)}
	\arrow[2]{e,b}{\psi({\overline P_1}{\cap}{\overline P}_{2},\phi({P}_{2},f))}
\node{}
	\arrow{e}
\node{M({\overline P}_{1}{\cap}{\overline P}_{2},f),}
\end{diagram}
\end{equation}
where the vertical arrows are $ P_1$-equivalences and the horizontal arrows are $ P_2$-equivalences. 

\begin{prop}
Assume that $M$ and the spaces $M(\overline P_i,f)$ for $i=1, 2$ admit fibrewise co-$H$-structures $\nu, \, \nu^{P_i}$ for  $i=1, 2$, respectively,  such that $\psi(\overline P_i,f)$ are  fibrewise co-$H$-maps, for $i=1, 2$. Then, 
$M ({\overline P}_{1}{\cap} {\overline P}_{2}, f)$ is a fibrewise co-$H$-space $(P_1 \cup P_2)$-equivalent to $X$.
\end{prop}
\begin{proof}
The commutativity of the diagram (\ref{diag:pushout}) gives a map from the pushout $W$ of $M({\overline P}_{1},f)$ and $M({\overline P}_{2},f)$ to $M({\overline P}_{1} \cap {\overline P}_{2},f)$, which induces an isomorphism of homology groups.
Thus it follows that $W$ and  $M({\overline P}_{1} \cap {\overline P}_{2},f)$ have the same fibre homotopy type.
Since a fibrewise pushout of fibrewise co-$H$-structures is a fibrewise co-$H$-space, we obtain that $M({\overline P}_{1} \cap {\overline P}_{2},f) $ is also a fibrewise co-$H$-space such that $\psi({\overline P}_{1} \cap {\overline P}_{2},f): M \rightarrow M({\overline P}_{1} \cap {\overline P}_{2},f)$ is a fibrewise co-$H$-map. 

Now, it suffices to consider 
$\phi(P_1, f): M({\overline P}_{1},f) \rightarrow X$ and $\phi(P_2, f): M({\overline P}_{2},f) \rightarrow X$, making the diagram commutative, to conclude that the fibrewise pushout is $(P_1 \cup P_2)$-equivalent to $X$.
\end{proof}

\begin{lemma}\label{lem:key}

Let $f : M \to X$ be the map constructed in Lemma \ref{lem:rational}, where $M$ is a wedge-sum of a finite number of spheres.  Suppose that $X^{B}_{({P}_{1})}$ and $X^{B}_{({P}_{2})}$ are  fibrewise co-$H$-spaces with fibrewise co-$H$-structures $\nu_{1}$ and $\nu_{2}$, respectively,  such that $\ell^{B}_{({P}_{1})}{\comp}f$ and $\ell^{B}_{({P}_{2})}{\comp}f$ are fibrewise co-$H$-maps with respect to co-$H$-structures on $M$, 
$\nu'$ and $\nu''$ respectively. Then, there exist fibrewise co-$H$-structures $\nu^{0}$ on $M$, $\nu^{{P}_{1}}$ on $X^{B}_{({P}_{1})}$ and $\nu^{{P}_{2}}$ on 
$X^{B}_{({P}_{2})}$ such that $\ell^{B}_{({P}_{1})}{\comp} {f}$ and $\ell^{B}_{({P}_{2})}{\comp} {f}$ are fibrewise co-$H$-maps with respect to 
$\nu^{0}$, $\nu^{{P}_{1}}$ and $\nu^{{P}_{2}}$ respectively. 
\end{lemma}
\begin{proof} Firstly remark that, to assume that there exist fibrewise co-$H$-structures on $M $ such that $\ell^{B}_{({P}_{1})}{\comp}f : M \rightarrow
X_{(P_1)}^B$ and $\ell^{B}_{({P}_{2})}{\comp}f: M \rightarrow
X_{(P_2)}^B$  are fibrewise co-$H$-maps, is equivalent by Proposition {\ref{prop:intermediatecoH}}  to assume that, there exist fibrewise co-$H$-structures on $M $ such that  
$\psi (\overline {P_1}, f): M \rightarrow M({\overline P_1}, f)$ and $\psi (\overline {P_2}, f): M \rightarrow M({\overline P_2}, f)$ are fibrewise co-$H$-maps.  To simplify notation, we will also refer to them  as $\nu^{{P}_{1}}$ on $M({\overline P_1}, f)$ and $\nu^{{P}_{2}}$ on 
$M({\overline P_2}, f) $. 

Now, since $M$ has the homotopy type of a wedge of spheres, there is a standard fibrewise co-$H$-structure $\nu$ on $M$ induced from the unique structure maps on spheres.
Let us denote by $d(\nu,\nu''), \, d(\nu,\nu'') : M \to \LoopB{M}\ast_{B}\LoopB{M}$ the fibrewise difference of the  fibrewise co-$H$-structure maps $\nu$, $\nu'$ and $\nu''$.

The proof will follow by induction on the homological decomposition of $M$.  Since the fibrewise differences $d(\nu,\nu')$, $d(\nu,\nu'')$ and $d(\nu',\nu'')$ are trivial on $M_{1}=M(X_{1}) = B$,  we put $\nu_{1}=\nu$, $\nu'_{1}=\nu'$, $\nu''_{1}=\nu''$, $\nu^{{P}_{1}}_{1}=\nu^{{P}_{1}}$ and $\nu^{{P}_{2}}_{1}=\nu^{{P}_{2}}$.

Secondly, let us assume that there are fibrewise co-$H$-structures 
 $\nu'_{i}$ and $\nu''_{i}$ on $M$, $\nu^{{P}_{1}}_{i}$ on $M({\overline P_1}, f)$ and $\nu^{{P}_{2}}_{i}$ on $M({\overline P_2}, f)$ such that $\psi (\overline {P_1}, f) $
  and $\psi (\overline {P_2}, f)$ are fibrewise co-$H$-maps. 
  We also assume that $d(\nu,\nu'_{i}) $ and $d(\nu,\nu''_{i})$ coincide on $M_{i}=M(X_{i})$ as induction hypotheses.
Since $\psi({\overline{{P}_{1}}},f)$ is $\overline{{P}_{1}}$-equivalence, there are extensions $d^{{P}_{1}}(\nu,\nu'_{i})$ and $d^{{P}_{1}}(\nu,\nu''_{i})$ of $sd(\nu,\nu'_{i})$ and $sd(\nu,\nu''_{i})$ on $M(\overline{{P}_{1}},f)$ for some $s$ such that $(s,\overline{{P}_{1}})=1$.
Similarly, we have extensions $d^{{P}_{2}}(\nu,\nu'_{i})$ and $d^{{P}_{2}}(\nu,\nu''_{i})$ of $td(\nu,\nu'_{i})$ and $td(\nu,\nu''_{i})$ on $M(\overline{{P}_{2}},f)$ for some $t$ such that $(t,\overline{{P}_{2}})=1$. 

Since ${P}_{1} \cap {P}_{2} = \emptyset$, we can choose integers $n$ and $m$ such that $ns+mt=1$. Let $\nu^{{P}_{1}}_{i+1}=\nu^{{P}_{1}}_{i}-nd^{{P}_{1}}(\nu,\nu'_{i})+nd^{{P}_{1}}(\nu,\nu''_{i})$ and $\nu^{{P}_{2}}_{i+1}=\nu^{{P}_{2}}_{i}+md^{{P}_{2}}(\nu,\nu'_{i})-md^{{P}_{2}}(\nu,\nu''_{i})$, where the sum is taken by $\nu^{{P}_{1}}_{i}$ and $\nu^{{P}_{2}}_{i}$ on 
$  M(\overline{{P}_{1}},f)$ and $ M(\overline{{P}_{2}},f)$, respectively.
Then $\nu^{{P}_{1}}_{i+1}$ and $\nu^{{P}_{2}}_{i+1}$ give fibrewise co-$H$-structures on $M(\overline{{P}_{1}},f) $ and $M(\overline{{P}_{2}},f) $.

Similarly, let $\nu'_{i+1}=\nu'_{i}-nsd(\nu,\nu'_{i})+nsd(\nu,\nu''_{i})$ and $\nu''_{i+1}=\nu''_{i}+mtd(\nu,\nu'_{i})-mtd(\nu,\nu''_{i})$, where the sum is taken by 
$\nu'_{i}$ and $\nu''_{i}$  on $M$, respectively. By construction,   $ \psi({\overline{{P}_{1}}},f)$ and $\psi({\overline{{P}_{2}}},f)$ are fibrewise co-$H$-maps with respect to the fibrewise co-$H$-structures $\nu'_{i+1}$ and $\nu''_{i+1}$ on $M$, and $\nu^{{P}_{1}}_{i+1}$ and $\nu^{{P}_{2}}_{i+1}$, respectively.

Now, using the fact that $ns+mt =1$, and also by classical arguments of connectivity, one can prove that $d(\nu, \nu'_{i+1}) $ coincides with $d(\nu, \nu''_{i+1}) $ over $M_{i+1}$.
Thus, by induction on $i$, we obtain the lemma.
\end{proof}
By Lemma \ref{lem:key}, we inductively obtain that $X$ is a fibrewise co-$H$-space, and therefore $\cover{X}$ is a co-$H$-space.
Hence, by Proposition \ref{prop:domination}, we obtain that $X$ is a co-$H$-space.
This completes the proof of Theorem \ref{thm:main}.

\begin{remark} Notice that  Theorem \ref{thm:main} cannot be directly obtained  from \cite[Theorem 1.1]{Costoya:loc_co-hopf}. Indeed, following along the lines of the previous paragraph, as $X^B_{(p)}$ is a fibrewise co-$H$-space, then ${\cover{X}}_{(p)}$ is a co-$H$-space for every prime $p$.  Since we are now in the simply-connected case, one could be tempted to use \cite{Costoya:loc_co-hopf} to conclude that $\cover{X}$
is a co-$H$-space (therefore $X$ is so, by Proposition \ref{prop:domination}).  Unfortunately, results in \cite{Costoya:loc_co-hopf}  are obtained for finite-type spaces while $\cover{X}$ is not.

\end{remark}
%
%

%
\end{document}